\documentclass{amsart}
\usepackage{amsmath,amscd}
\usepackage{amsfonts}
\usepackage{amssymb}
\usepackage{graphicx}
\usepackage{amsthm}
\usepackage{newlfont}

\newcommand{\f}{\frac}

\newcommand{\ds}{\displaystyle}

 \newtheorem{thm}{Theorem}[section]
 \newtheorem{cor}[thm]{Corollary}
 \newtheorem{lem}[thm]{Lemma}
 \newtheorem{prop}[thm]{Proposition}

 \theoremstyle{remark}
 \newtheorem{rem}[thm]{Remark}
\newtheorem{oq}[thm]{Open Question}
  
 \numberwithin{equation}{section}
\begin{document}

\title[Some considerations on the nonabelian tensor square of crystallographic groups]
 {Some considerations on the nonabelian tensor square of crystallographic groups}
%----------Author 1
\author[A. Erfanian]{Ahmad Erfanian}

\address{Mathematics Department and \\ Centre of Excellence in Analysis
on Algebraic Structures\\ Ferdowsi University of Mashhad\\
P.O.Box 1159, 91775, Mashhad, Iran} \email{erfanian@math.um.ac.ir}

%\thanks{We thank Prof. B. Eick, who suggested \cite{en, eick},  Prof. R. F. Morse and Dr. P. Moravec,
%who communicated some inaccuracies in the original version of the present paper. Finally, we appreciated some
%email contributions of R. Brown, A. Caranti, D. Feirtenschlager, R. Hartung, M. Horn and D. Ramras in 2010.}

%\author[R. Masri]{Rohaidah Masri}
%\address{Mathematics Department, University Technology of Malaysia\\
%81310, Skudai, Johor, Malaysia} \email{rohaidah@upsi.edu.my}

\author[F. G. Russo]{Francesco G. Russo}
\address{Laboratorio di Dinamica e Geotecnica - Strega\\
Universit\'a del Molise\\
via Duca degli Abruzzi, Termoli (CB)} \email{francescog.russo@yahoo.com}

\author[N.H. Sarmin]{Nor Haniza Sarmin}

\address{Department of Mathematics, Faculty of Science\\
Universiti Teknologi Malaysia\\
81310 UTM Johor Bahru, Johor, Malaysia} \email{nhs@utm.my}

%\subjclass{}


\date{\today}
\dedicatory{}
%%% ----------------------------------------------------------------------

\begin{abstract}
The nonabelian tensor square $G\otimes G$ of a polycyclic group $G$ is  a polycyclic group and its structure
arouses interest in many contexts. The same assertion is still true for wider classes of solvable groups. This
motivated us to work on two levels in the present paper: on a hand, we investigate the growth of the Hirsch
length of $G\otimes G$ by looking at that of $G$, on another hand, we study the nonabelian tensor product of
pro--$p$--groups of finite coclass, which are a remarkable class of  solvable groups without center, and then we
do considerations on their Hirsch length. Among other results, restrictions on the Schur multiplier will be
discussed.
\end{abstract}

%%% ----------------------------------------------------------------------
\maketitle
%%% ----------------------------------------------------------------------

\section{Introduction}

The $nonabelian$ $tensor$ $square$ $G\otimes G$ of a group $G$ is the group generated by the symbols $x\otimes
y$ and subject to the relations \begin{equation} xy\otimes z=(^xy\otimes \hspace{.04cm} ^xz)(x\otimes z) \ \ \
\textrm{and} \ \ \ x\otimes zt=(x\otimes z)\ (^zx\otimes\hspace{.04cm}^zt )\end{equation}  for all $x,y,z,t\in
G,$ where $G$ acts on itself via conjugation $^xy=xyx^{-1}$. In particular, if $G$ is abelian and acts trivially
on itself, we have the usual abelian tensor product $G \otimes_\mathbb{Z} G$. The group $G\wedge G=G\otimes
G/\nabla(G)$ is called $nonabelian$ $exterior$ $square$ of $G$, where $\nabla(G)=\langle x\otimes x \ | \ x\in
G\rangle$. From \cite{bjr} the maps $\kappa: x\otimes y\in G\otimes G\longmapsto [x,y]\in [G,G]$ and $\kappa':
x\wedge y \in G\wedge G \longmapsto [x,y]\in [G,G]$ are epimorphisms and the topological meaning of
$\ker\kappa=J_2(G)$ is described in \cite{bl}. Still from \cite{bjr} the following diagram is commutative with
exact rows and central extensions as columns:
\begin{equation}\label{e:graph1}
\begin{CD}
@. @.  1 @. 1\\
@. @. @VVV   @VVV\\
H_3(G)@>>>\Gamma(G^{ab}) @>\psi>>J_2(G)@>>>H_2(G)@>>>1\\
@| @| @VVV @VVV\\
H_3(G)@>>>\Gamma(G^{ab}) @>\psi>>G\otimes G @>>>G\wedge G@>>>1 @. \\
@. @. @V\kappa VV   @V\kappa'VV\\
@. @. [G,G] @= [G,G]\\
@. @. @VVV   @VVV\\
@. @. 1 @. 1\\
\end{CD}
\end{equation}
where $H_2(G)=\ker \kappa'$ is the second integral homology group of $G$, $H_3(G)$ is the third integral
homology group of $G$ and $\Gamma$ is the Whitehead's quadratic functor in \cite[Section 2]{bjr}. We note that
$H_2(G)$ is exactly the \textit{Schur multiplier} $M(G)$ of $G$.

After the initial work \cite{d}, many authors investigated the structure of $G\otimes G$ by looking at that of
$G$ and in the last times there is a significant production which is devoted to the classes $\mathfrak{P}$ of
all polycyclic groups, $\mathfrak{F}$ of all finite groups  and $\mathfrak{S}$ of all solvable groups (see
\cite{bfm, bm, en, m, visscher}). In a solvable group $G$ we recall that the number of infinite cyclic factors
$h(G)$ is an invariant, called \textit{Hirsch length}, or \textit{torsion--free rank}, of $G$ (see \cite[pp.14,
15, 16, 85]{lr}). If $G\in \mathfrak{P}$, we have $h(G)=0$ if and only if $G \in \mathfrak{P}\cap \mathfrak{F}$.
Now, if $G$ is abelian, then $G \otimes G$ is abelian by \cite[Theorem 3.1]{visscher}; if $G\in \mathfrak{P}$,
then $G\otimes G\in \mathfrak{P}$ (see \cite{bm, en, m}) and, so far as we know, the structure of $G\otimes G$
is widely described in terms of the upper central series of $G$. For instance, \cite{ksv} classifies $G\otimes
G$, when $G$ is a 2--generator 2--group, and so, $G$ is a particular type of polycyclic group with nontrivial
center. \cite{niza} describes $G\otimes G$, where $G$ is an infinite nonabelian 2--generator nilpotent group of
class 2, and so, $G$ is still a polycyclic group with nontrivial center. There are several contributions on this
line of research but it is hard to find information on $G\otimes G$ when $G$ is a polycyclic group with trivial
center: we found the initial idea in \cite{bk} and a recent interest in \cite{bfm,en,eick,rohaidah}.

The aim of the present work is to detect the structure of $G\otimes
G$, when $G$ is a polycyclic group with trivial center, or more
generally an infinite solvable group with trivial center, starting
from bounds on $h(G\otimes G)$ and $h(G)$. The absence of literature
on such a line of investigation has motivated us to write the
present paper. On another hand, R. F. Morse  has kindly pointed out
(see \cite{bob}) that the same question was posed by C. Rover at the
Conference on Computational Group Theory and Cohomology at the
Harlaxton College (Harlaxton Lincolnshire, UK) in 2008. We end this
introduction, noting that the terminology and the notations of the
present paper are standard and can be found in \cite{bjr, bl, d,
ellis2}.

\section{The growth of the Hirsch length in the nonabelian tensor square}

The following (unpublished) lemma was communicated to us by D.
Ramras and describes some classical situations, which we may
encounter, when we deal with the presentations of polycyclic groups.
Further details can be found in \cite{ramras}.

\begin{lem}\label{l:ramras}Let $l,p,k,m,n_1, n_2, \ldots ,n_m$ be integers.
Consider an extension of groups $1 \rightarrow A \rightarrowtail
\Gamma \twoheadrightarrow Q \rightarrow 1$ in which $A$ is a
finitely generated abelian group and $Q$ is finite. If
\begin{equation}Q =\langle q_1,\ldots,q_l \ | \ r_1(q_1,\ldots,q_l)
= \ldots = r_p(q_1,\ldots,q_l) = 1 \rangle\end{equation} for some
words $r_1,\ldots,r_p$ in the free group on $l$ letters and
\begin{equation}A=\langle a_1,\ldots,a_{k+m} \ | \ [a_i,a_j] = 1 \
(1\le i\le j \le k+m), \ a^{n_1}_1=\ldots =a^{n_m}_m=1 \ (1\le
n_1\leq \ldots\leq n_m)\rangle,\end{equation} then for some words
$w_i$ and $u_{ij}$ (not uniquely determined) in the free group on $k
+ m$ letters,
\begin{equation}\Gamma= \langle \alpha_1,\ldots,\alpha_{k+m}, \gamma_1, \ldots, \gamma_l \ | \
r_1(\gamma_1,\ldots, \gamma_l) = w_1(\alpha_1,\ldots,\alpha_{k+m}),
\ldots, \end{equation}\begin{equation}r_p(\gamma_1,\ldots, \gamma_l)
= w_p(\alpha_1,\ldots,\alpha_{k+m}), [\alpha_i,\alpha_j]=1,
\alpha^{n_j}_j=1, \gamma_i
\alpha_j\gamma_i^{-1}=u_{ij}(\alpha_1,\ldots, \alpha_{k+m}),
\end{equation}\begin{equation}\ (1\le i\le j \le k+m)
\rangle.\end{equation}
\end{lem}

\begin{proof}
To begin, we must specify the words $u_{ij}$ and $w_i$. Choose
elements $\widetilde{q}_i \in \Gamma$ lying over $q_i \in Q$. Since
$A$ is normal in $\Gamma $, we know that
$\widetilde{q}_ia_j\widetilde{q}^{-1}_i \in A$, and hence
$\widetilde{q}_ia_j\widetilde{q}^{-1}_i=u_{ij}(a_1,\ldots a_{k+m})$
for some word $u_{ij}$. Next, since $r_i(q_1,\ldots,q_l) = 1$ in
$Q$, we know that $r_i(\widetilde{q}_1,\ldots, \ldots{q}_l)\in A$,
and hence $r_i(q_1,\ldots,q_l) = w_i(a_1,\ldots,a_{k+m})$ for some
word $w_i$. Now, let $\widetilde{\Gamma}$ denote the group presented
by (2.3)--(2.5), and let $\widetilde{A}$ denote the subgroup
generated by $\alpha_1,\ldots, \alpha_{k+m}$. Let $\Phi
:\widetilde{\Gamma}\rightarrow \Gamma$ be the homomorphism defined
by $\Phi(\alpha_i) = a_i$ and $\Phi(\gamma_i) = \widetilde{q}_i$.
Then $\Phi$ is surjective, and its restriction to $\widetilde{A}$ is
a surjection onto $A \le \Gamma$. The third set of relations in
(2.3)--(2.5) ensures that $\widetilde{A}$ is normal in
$\widetilde{\Gamma}$, and we define $\widetilde{Q} =
\widetilde{\Gamma}/\widetilde{A}$. The map $\Phi$ induces a
surjection $\widetilde{\Phi}: \widetilde{Q} \twoheadrightarrow Q$,
and we have a commutative diagram
\begin{equation}\label{e:graphramras}
\begin{CD}
\ 1 @>>>\widetilde{A} @>>>\widetilde{\Gamma}@>>>\widetilde{Q}@>>>1\\
@. @VVV @V\Phi VV @V\widetilde{\Phi} VV\\
\ 1 @>>>A @>>>\Gamma@>>>Q@>>>1\\
%@.  @V\kappa VV   @V\kappa'VV\\
%@. @. [G,G] @= [G,G]\\
%@. @. @VVV   @VVV\\
%@. @. 1 @. 1\\
\end{CD}
\end{equation}
The map $\widetilde{\Gamma}\rightarrow \widetilde{Q}$ induces a
surjection from the free group on the generators $\gamma_i$ onto
$\widetilde{Q}$, and this surjection factors through the quotient
group $\langle \gamma_1, \ldots \gamma_l \ | \  r_i(
\gamma_1,\ldots, \gamma_l) = 1 \rangle \simeq Q$. Hence we have a
surjection $Q \twoheadrightarrow \widetilde{Q}$, meaning that
$\widetilde{Q}$ is a finite group of order at most $|Q|$. The
existence of the surjection $\widetilde{\Phi} : \widetilde{Q}
\twoheadrightarrow Q$ now shows that both of these surjections must
in fact be isomorphisms. Next, we show that the map $\widetilde{A}
\rightarrow  A$ is injective. Each element $\alpha \in
\widetilde{A}$ has the form $\alpha^{p_1}_1 \alpha^{p_2}_2 \ldots
\alpha^{p_{k+m}}_{k+m}$ for some integers $p_i$. Our presentation
for $A$ shows that, if $\Phi( \alpha) = 0$, then $p_i$ is a multiple
of $n_i$ for $1\le i \le m$, and $p_i = 0$ for $i > m$. But such
elements are already trivial in $\widetilde{\Gamma}$, so $\phi$ is
injective when restricted to $\widetilde{A}$. We have now shown that
the two outer maps in (\ref{e:graphramras}) are isomorphisms, and
the 5-lemma shows that $\Phi$ is an isomorphism as well.
\end{proof}

Lemma \ref{l:ramras} can be specialized in various ways. For instance, assume that the cyclic group $C_n=\langle
t \ | \ t^n=1\rangle$ of order $n>1$ is equal to $Q$;  the free abelian group
$\mathbb{Z}^{n-1}=\underbrace{\mathbb{Z} \times \ldots \times \mathbb{Z}}_{(n-1)-\textrm{times}}=\langle
a_1,\ldots,a_{n-1} \ | \ [a_i,a_j]=1; 1\leq i,j \leq n-1\rangle $ of rank $n-1$ is equal to $A$; $C_n$ acts on
$\mathbb{Z}^{n-1}$ via the following homomorphism
\begin{equation}\label{1tris}\xi : t\in C_n \mapsto
\xi(t)=\left( \begin{array}{ccccccc}
0 & 1 & 0 \ldots & 0 & 0\\
0 & 0 & 1 \ldots & 0 & 0\\
\ldots & \ldots &  \ldots & \ldots & \ldots\\
0 & 0 & 0 \ldots & 0 & 1\\
-1 & -1 & -1 \ldots & -1 & -1\\
\end{array} \right)\in GL_{n-1}(\mathbb{Z}).\end{equation}
We have the \textit{crystallographic group} $G_n=C_n\ltimes \mathbb{Z}^{n-1}$ $of$ $holonomy$ $n$ and several
information on  it can be found in \cite[\S 6.3]{en}, or \cite[Proposition 3.3]{agpp}.
 Looking at its construction,  $G_n\in \mathfrak{P}$, $h(G_n)=n-1$,  $Z(G_n)=\{1\}$ and
$G_n$ is metabelian (in particular, $[G_n,G_n]$ is abelian). On another hand, we may get a presentation for
$G_n$, taking a generating set for $C_n$, another for $\mathbb{Z}^{n-1}$ and considering the action
(\ref{1tris}). We have as follows.

\begin{cor}
$G_n=\langle a_1,\ldots,a_{n-1},t \ | \ t^n=1, t^{-1}a_it=a_{i+1} \ (1 \le i \le n-2), t^{-1}a_{n-1}t= a_1^{-1}
\ldots a_{n-1}^{-1}, [a_i,a_j]=1 \ (1\le j<i \le n-1) \rangle.$
\end{cor}
We can be more accurate in the description of $[G_n,G_n]$ and of the abelianization $G^{ab}_n=G_n/[G_n,G_n]$. In
fact $[G_n,G_n]$ is generated by the commutators of generators of $G_n$ and their inverses. Since all the $a_i$
commute and $t$ has finite order, one has only to consider commutators of the form $[t,a_i]$  and thus
\begin{equation}[G_n,G_n]=\langle a^{-1}_ia_{i+1}, a^{-1}_1\ldots a^{-1}_{n-2}a^{-2}_{n-1} \ |\end{equation}\begin{equation} \ [a^{-1}_ia_{i+1},a^{-1}_ja_{j+1}]=[a^{-1}_ia_{i+1},a^{-1}_1\ldots a^{-1}_{n-2}a^{-2}_{n-1}]=1 \ (1 \le i,j \le
n-2)\rangle.
\end{equation}
We note that $[G_n,G_n]$ is free abelian of rank $n-1$. On another hand, when we factor $G_n$ through
$[G_n,G_n]$, we have that $t$ is an independent generator and $a_1=a_2= \ldots=a_{n-1}$. So
$a_{n-1}=a_{n-1}^{-(n-1)}$ which implies $a^n_{n-1}=1$ and $a_{n-1}$ is a second independent generator. We
conclude that $G^{ab}_n=C_n \times C_n$.

On another hand,  if $G_n$ has $n=p^s$ ($p$  prime and $s\geq1$) and we replace $\mathbb{Z}^{p-1}$ with
$\mathbb{Z}^{d_s}_p$, where $\mathbb{Z}_p$ denotes the $group$ $of$ $p$--$adic$ $integers$ and
$d_s=p^{s-1}(p-1)$, then we have the $pro$--$p$--$group$ $K_s=C_{p^s}\ltimes \mathbb{Z}^{d_s}_p$ $of$ $finite$
$coclass$  $with$ $central$ $exponent$ $s$, studied in \cite{eick, charles}. This time we cannot apply Lemma
\ref{l:ramras}, but computational arguments are still true. We recall a result in this direction, to convenience
of the reader.

\begin{lem}[See \cite{eick},  Theorem 7]\label{l:eick} For an integer $i$ let $e_i=1$, if $p^{s-1}$ divides
$i-1$, and $e_i=0$, otherwise. Then $ K_s= \langle a_1,\ldots,a_{d_s},t \ | \ t^{p^s}=1, \
t^{-1}a_1t=a^{-1}_{d_s}, \ t^{-1}a_it=a_{i-1}a^{-e_i}_{d_s} \ (1<i\le d_s), [a_i,a_j]=1 \ (1\leq j<i\le
d_s)\rangle$. Furthermore, $M(K_s)\simeq \mathbb{Z}^{ \frac{d_s}{2}}_p$, unless $p=2$ and $s=1$ in which case
$M(K_s)=1$.
\end{lem}

We may use the above arguments in order to note that $K_s$ is a metabelian group with $h(K_s)=d_s$, $[K_s,K_s]
\simeq \mathbb{Z}_p^{d_s}$, $K^{ab}_s=C_{p^s} \times C_{p^s}$ and $Z(K_s)=\{1\}$. However, $K_s\not \in
\mathfrak{P}$, but $K_s\in \mathfrak{S}$.

B. Eick and W. Nickel \cite{en} have studied the nonabelian tensor square of $G_n$, when $n=p$.  For $p=2$ we
have the infinite dihedral group $G_2=D_\infty=\langle a,x \ | \ a^x=a^{-1}, x^2=1\rangle=C_2\ltimes
\mathbb{Z}$. Quoting \cite[Figure at p.943]{en}, the following list holds:
\begin{equation}\label{e:1}
h(G_2\otimes G_2)=h(G_2)=1,  h(G_3\otimes
G_3)-h(G_3)=3-2=1,\end{equation}\begin{equation} h(G_5\otimes
G_5)-h(G_5)=6-4=2, h(G_7\otimes G_7)-h(G_7)=9-6=3,
\ldots.\end{equation} With the help of GAP \cite{gap} one can see
that the same list is true when $s=1$, $p=2,3,5,7$ and we deal with
$K_2=C_2\ltimes \mathbb{Z}_2, K_3=C_3 \ltimes \mathbb{Z}^2_3,
K_5=C_5 \ltimes \mathbb{Z}^4_5, K_7=C_7\ltimes \mathbb{Z}^6_7$. Then
it would be interesting to detect the properties of the following
function from the set of the integers onto the set of the integers
\begin{equation}\label{e:2}f: h(S)\in \{h(S) \ | \ S \in \mathfrak{S}\} \mapsto f(h(S))=h(S\otimes S)-h(S).\end{equation}

\begin{rem}\label{r:1}
I. Nakaoka and M. Visscher show that $S\otimes S\in \mathfrak{S}$,
whenever $S\in \mathfrak{S}$ (see \cite{bm,m,visscher}) and so $f$
is well--posed. On another hand, G. Ellis \cite{ellis1} and P.
Moravec \cite{m} show that $F\otimes F\in \mathfrak{F}\cap
\mathfrak{P}$, whenever $F\in \mathfrak{F}\cap \mathfrak{P}$. Then
$0=h(C_2)\mapsto f(0)=0$, or more generally, $0=h(F)\mapsto f(0)=0$
for all $F\in \mathfrak{F}\cap \mathfrak{P}$, but also
$1=h(G_2)\mapsto f(1)=0$. Hence $f$ is not injective. In fact
$N(f)=\{h(S) \ | \ f(h(S))=0\}=\{h(S\otimes S)=h(S) \ | \ S\in
\mathfrak{S} \}$. Finally, one can note that $f$ is neither additive
nor multiplicative.
\end{rem}

%In the next figure there are some computational data, related to
%(\ref{e:2}) and $G_p$. The corresponding figure for $K_p$ is
%omitted, since the values are exactly the same of the case of $G_p$.

%\includegraphics{prova1.tif}

The next property of the Hirsch length is well--known.

\begin{lem}[See \cite{lr}, \S 1.3]\label{l:1}
If $A,B \in \mathfrak{S}$ and $\varphi: A \rightarrow B$ is a homomorphism of groups, then
$h(A)=h(\varphi(A))+h(\ker \varphi)$. In particular, the Hirsch length is additive on the extensions.
\end{lem}

We have immediately the next consequence.

\begin{cor}\label{c:1extra}$f(h(S))\leq
h(J_2(S))$ for all $S \in \mathfrak{S}$.
\end{cor}

\begin{proof} (\ref{e:graph1}) shows that $S\otimes S\in
\mathfrak{S}$ is a central extension of $J_2(S)$ by $[S,S]$. From
Lemma \ref{l:1}, $h(S\otimes S)=h(J_2(S))+h([S,S])$. On another
hand, $[S,S]\leq S$ implies $h([S,S])\leq h(S)$ and so $h(S\otimes
S)\leq h(J_2(S))+h(S)$ from which the result follows.
\end{proof}

We recall the following information on the structure of $J_2(G)$, $\nabla(G)$ and $G\otimes G$.

\begin{prop}[See \cite{bfm}, Corollary 1.4]\label{p:bfm1}
Let $G$ be a group such that $G^{ab}$ is abelian finitely generated with no elements of square order. Then
$J_2(G)=\Gamma(G^{ab})\times M(G)$.
\end{prop}

\begin{prop}[See \cite{bfm}, Theorem 1.3 (iii)]\label{p:bfm2}
Let $G$ be a group such that either $G^{ab}$ has no elements of square order or $G'$ has a complement in $G$.
Then $\nabla(G)\simeq \nabla(G^{ab})$ and $G\otimes G \simeq \nabla(G) \times (G \wedge G)$.
\end{prop}

The linear growth of (\ref{e:2}) is described by the next result.

\begin{prop}\label{p:2} In Lemma \ref{l:eick} let  $s=1$, $p\not=2$ and $K_p=C_p\ltimes \mathbb{Z}^{p-1}_p$ be the corresponding pro--$p$--group.
Then $f(h(K_p))=\frac{1}{2}(p-1)$. In particular,
$f(h(K_p))=h(J_2(K_p))=h(M(K_p))$ has a linear growth.
\end{prop}

\begin{proof} We claim that (\ref{e:graph1}) is equivalent to the
following diagram
\begin{equation}\label{e:graph2}
\begin{CD}
@. @.  1 @. 1\\
@. @. @VVV   @VVV\\
H_3(K_p)@>>>C^2_p \times C_{p^2}@>\psi>>C^2_p \times C_{p^2} \times \mathbb{Z}_p^{\f{p-1}{2}}@>>>\mathbb{Z}_p^{\f{p-1}{2}}@>>>1\\
@| @| @VVV @VVV\\
H_3(K_p)@>>>C^2_p \times C_{p^2} @>\psi>>C^2_p \times C_{p^2} \times
\mathbb{Z}^{p-1}_p \times
\mathbb{Z}^{\f{p-1}{2}}_p@>>>\mathbb{Z}^{p-1}_p  \times \mathbb{Z}^{\f{p-1}{2}}_p @>>>1.\\
@. @. @V\kappa VV   @V\kappa'VV\\
@. @. \mathbb{Z}_p^{p-1} @= \mathbb{Z}_p^{p-1}\\
@. @. @VVV   @VVV\\
@. @. 1 @. 1\\
\end{CD}
\end{equation}
From \cite[\S2, (13), p.181]{bjr}, \begin{equation}\Gamma(K_p^{ab})=\Gamma(C_p \times C_p)=C_p \times C_p \times
(C_p \otimes_\mathbb{Z} C_p)=C_p \times C_p \times C_{p^2}.\end{equation} Note that $C_p \otimes_\mathbb{Z} C_p
=C_{p^2}$ is an elementary fact on the usual abelian tensor product. Still by \cite[\S2]{bjr},
\begin{equation}\psi(\Gamma(C_p \times C_p))=\nabla(K_p)=C_p \times C_p \times C_{p^2}.\end{equation} From Lemma \ref{l:eick},
$M(K_p)=\mathbb{Z}_p^{\f{p-1}{2}}$. We do not have elements of square order in $K_p^{ab}=C_p \times C_p$ and
Proposition \ref{p:bfm1} yields $J_2(K_p)\simeq \Gamma(K_p^{ab})\times M(K_p) \simeq C_p \times C_p \times
C_{p^2} \times \mathbb{Z}_p^{\f{p-1}{2}}$.

The commutativity of (\ref{e:graph1}) shows that $K_p\wedge K_p$ is
a central extension of $M(K_p)=\ker \kappa'$ by $[K_p,K_p]$, which
are both normal abelian subgroups of $K_p\wedge K_p$, then
$K_p\wedge K_p=\langle
M(K_p),[K_p,K_p]\rangle=M(K_p)[K_p,K_p]=\langle \mathbb{Z}_p^{p-1},
\mathbb{Z}_p^\frac{p-1}{2}\rangle= \mathbb{Z}_p^{p-1}
\mathbb{Z}_p^\frac{p-1}{2}.$ On another hand,
\begin{equation}[M(K_p),M(K_p)]=[[K_p,K_p],[K_p,K_p]]=1\end{equation} implies
\begin{equation}[K_p\wedge K_p,K_p\wedge
K_p]=[M(K_p)[K_p,K_p],M(K_p)[K_p,K_p]]=[M(K_p),M(K_p)][M(K_p),[K_p,K_p]]\end{equation}\begin{equation}[[K_p,K_p],[K_p,K_p]]
[M(K_p),[K_p,K_p]]= [M(K_p),[K_p,K_p]].\end{equation} Since $M(K_p)=\mathbb{Z}_p^\frac{p-1}{2}\leq
\mathbb{Z}_p^{p-1}=[K_p,K_p]$,  we deduce $C_{K_p\wedge K_p}([K_p,K_p])\leq C_{K_p\wedge K_p}(M(K_p))$ and then
\begin{equation}[K_p,K_p]\leq C_{K_p\wedge K_p}([K_p,K_p])\leq C_{K_p\wedge K_p}(M(K_p)),\end{equation} which
implies $[M(K_p),[K_p,K_p]]=1$. We conclude that $K_p \wedge K_p$ is abelian and then the central extension is
actually a direct product of the form $K_p \wedge K_p = \mathbb{Z}^{p-1}_p  \times \mathbb{Z}^{\f{p-1}{2}}_p$.
From Proposition \ref{p:bfm2},
\begin{equation}\label{e:crystallographic2}K_p\otimes K_p=C_p \times C_p \times C_{p^2} \times \mathbb{Z}^{p-1}_p  \times \mathbb{Z}^{\f{p-1}{2}}_p.\end{equation}
Then
\begin{equation}h(M(K_p))=h(J_2(K_p)/\nabla(K_p))=h(J_2(K_p))=\frac{p-1}{2}.\end{equation}
We conclude from (\ref{e:graph2}) and Lemma \ref{l:1} that
\begin{equation}\label{e:extra}h(K_p\otimes K_p)=h(\kappa(K_p\otimes
 K_p))+h(J_2(K_p))=(p-1)+h(M(K_p))=\frac{3}{2}(p-1).\end{equation} Therefore
 $f(h(K_p))=h(J_2(K_p))=\frac{1}{2}(p-1)$.
\end{proof}

The methods in the above proof continue to be valid when $s>1$. Therefore we draw the following result, which
has independent interest and, in view of \cite[Theorem 7.4.12, Corollary 7.4.13]{charles}, describes the
nonabelian tensor square of all pro--$p$--groups of finite coclass with trivial center.

\begin{thm}\label{t:propgroups}
If $s>1$ and $p$ is an odd prime, then $K_s \otimes K_s=C^2_{p^s}
\times C_{p^{2s}}\times \mathbb{Z}^{\frac{3}{2}d_s}_p$.In
particular, $f(h(K_s))=h(J_2(K_s))=h(M(K_s))$ has a linear growth.
\end{thm}

\begin{proof} Mutatis mutandis, we may argue as in the proof of Proposition \ref{p:2}.
\end{proof}

The computational data show that $M(G_p)=\mathbb{Z}^{\f{p-1}{2}}$. In alternative, an argument as in \cite[Proof
of Theorem 7]{eick} can be applied, that is, we may express the Schur's Formula for $M(G_p)$, beginning from the
presentation in Corollary \ref{c:1extra}. Equivalently, we may work via duality, since the cohomology of $G_p$
is known by \cite{agpp}. This justifies the assumption of the next result.

\begin{cor}\label{c:extra} Assume $M(G_p)=\mathbb{Z}^{\f{p-1}{2}}$ for all primes
$p\not=2$. Then $f(h(G_p))=\frac{1}{2}(p-1)$. In particular,
$f(h(G_p))=h(J_2(G_p))=h(M(G_p))$  has a linear growth.
\end{cor}

\begin{proof}
We may argue as in the proof of Proposition \ref{p:2}, replacing $K_p$ with $G_p$.
\end{proof}

The above results prove that there are crystallographic groups of
holonomy $p\not=2$ which achieve the bound in Corollary
\ref{c:1extra}. The same is true for the pro--$p$--group $K_p$ with
$p\not=2$. Note that Proposition \ref{p:2} describes rigorously the
structure of $K_p\otimes K_p$ with respect to that of $K_p$ in terms
of their torsion--free factors. The same is true for $G_p$ by
Corollary \ref{c:extra}. The fact that (\ref{e:2}) has a linear
growth can be translated in terms of restrictions on the Schur
multiplier as follows.

\begin{cor}\label{c:chissa1}If $f(h(S))=c \ h(S)\ $ for some integer $c\geq 0$ and $S \in
\mathfrak{S}$, then $h(M(S))\leq h(S)^2+(c+1)h(S)$. The equality
holds, whenever $S \in \mathfrak{F}$.
\end{cor}

\begin{proof} We have
$f(h(S))=h(S\otimes
S)-h(S)=\Big(h(J_2(S))+h([S,S])\Big)-h(S)=h(M(S))-h(\nabla(S))+h([S,S])-h(S)$.
Now we may always write $s\otimes s=(s\otimes 1) (1 \otimes s) $ in
a unique way and then the map $\iota:  s\otimes s \in \nabla(S)
\mapsto \iota(s\otimes s)=\iota((s\otimes 1)(1 \otimes
s))=\iota(s\otimes 1) \iota(1 \otimes s)=(s,s) \in S \times S$  is a
monomorphism. Therefore $h(\nabla(S))\leq h(S)^2$ and so
$h(M(S))=f(h(S))+h(\nabla(S))-h([S,S])+h(S)\leq
f(h(S))+h(\nabla(S))+h(S) \leq c \ h(S)+h(S)^2+h(S)$ from which the
result follows.
\end{proof}

Unfortunately, (\ref{e:2}) has not a linear growth for all groups in
$\mathfrak{S}$ and we cannot predict its form in general. Already in
$\mathfrak{P}$ there are examples in this sense (see \cite[Figure at
p.943]{en}). However, a nice circumstance is described below.

\begin{cor}
%\begin{itemize}
%\item[(i)]
There exists a metabelian group $G$ with trivial center for which
$f(h(G))=h(M(G))=\frac{1}{2}p^{s-1}(p - 1)$, where $s>1$ and $p$ is
an odd prime.
%\item[(ii)] Let $p$ be an odd prime. There is a sequence of
%$\{B_p\}_{p\in \mathbb{P}}$ in $\mathfrak{S}$ such that $f=\frac{1}{2k}r$, where $r, k\geq 1$.
%\end{itemize}
\end{cor}

\begin{proof}
  Consider $G=K_s$. By Lemma \ref{l:eick}, $h(M(K_s)) = \frac{1}{2}p^{s-1}(p - 1)$. From Theorem \ref{t:propgroups},
$f(h(K_s))=h(K_s\otimes K_s)-h(K_s)=\big(p^{s-1}(p - 1)+
\frac{1}{2}p^{s-1}(p - 1)\big)-p^{s-1}(p - 1)=\frac{1}{2}p^{s-1}(p -
1)=h(M(K_s)).$
%(ii). For $k=1$,  the result follows from Proposition \ref{p:2}.
%Assume $k\geq2$ and choose a prime $p\equiv1 \mod k$. The group
%$B_p=C_p\ltimes \mathbb{Z}^{\frac{1}{k}(p - 1)}$ has
%$h(B_p)=\frac{1}{k}(p - 1)$. Combining \cite[Theorem A]{eick} and
%Proposition \ref{p:2}, we get
%$h(M(B_p))=\frac{1}{2}\big(\frac{1}{k}(p - 1)\big)=\frac{1}{2k}(p -
%1)$ and so $f(h(B_p))=\frac{1}{2k}(p - 1)=\frac{1}{2k}r$, where
%$r=p-1\geq1$.
\end{proof}

We end the section with an explicit description for (\ref{e:2}),
modifying a classic argument of N. Rocco, which can be found in
\cite[Theorem 1]{bfm} (see also \cite[Observation]{bfm}).

\begin{thm}\label{t:1}Let $G$ be a group in $\mathfrak{P}$ such that
$G^{ab}=\ds\prod_{i=1}^nC_{p^{e_i}}$, for integers $1\leq e_i\leq
e_j$ such that $1\leq i <j \leq n$, $p$ odd prime and
$d=\sum_{i=1}^n(n-i)e_i$.
\begin{itemize}\item[(a)]If $G$ is finite, then $|G\otimes
G|=p^d|G||M(G)|$.
\item[(b)]If $G$ is infinite, then $f(h(G))=h(M(G))$.\end{itemize}\end{thm}

\begin{proof}
(a). Assume $G$ is finite. Since $G^{ab}$ is finitely generated and
has no elements of order two, all the hypotheses of \cite[Theorem
1]{bfm} are satisfied and so $G\otimes G\simeq \nabla(G) \times
G\wedge G$. From this and (\ref{e:graph1}) we deduce
\begin{equation}\label{e:5}|G\otimes
G|=\f{|\nabla(G)|}{|G^{ab}|}|G||M(G)|=\f{|\Gamma(G^{ab})|}{|G^{ab}|}|G||M(G)|=|\prod_{i=1}^n(C_{p^{e_i}})^{n-i}||G||M(G)|=p^d|G||M(G)|\end{equation}
where $d=\sum_{i=1}^n(n-i)e_i.$

(b). Assume $G$ is infinite. From Proposition \ref{p:bfm1} and Lemma \ref{l:1} we conclude that
$h(J_2(G))=h(\Gamma(G^{ab}))+ h(M(G))=h(M(G))$, where the last equality is due to the fact that $\Gamma(G^{ab})$
is periodic. Proceeding as in (\ref{e:extra}),
\begin{equation}h(G\otimes G)=h(\kappa(G\otimes G))+
h([G,G])=h(J_2(G))+ h([G,G])=h(M(G))+h([G,G]).\end{equation}
Subtracting $h(G)$, we find \begin{equation} f(h(G))=h(G\otimes
G)-h(G)=h(M(G))-(h(G)-h([G,G]))=h(M(G))-h(G^{ab})=h(M(G)),
\end{equation} since $G^{ab}$ is periodic.
\end{proof}

\begin{rem} \label{r:3}
It is not used the hypothesis $G\in \mathfrak{P}$ in Theorem
\ref{t:1} (a) and so this part of the result is true for an
arbitrary finite group.
\end{rem}

%The following consequence is straightforward, thanks to Corollary \ref{c:1}.

%\begin{cor}\label{c:4}
%All the polycyclic groups satisfying Theorem \ref{t:1} (b) have $f$ which is constant to 0.
%\end{cor}

\section{Some evidences}
The present section is devoted to  evaluate (\ref{e:2}) for other classes of groups for which it is known their
nonabelian tensor product. A $Bieberbach$ $group$ $B$ is an extension of a free abelian group $L$ (called
$lattice$ $group$) of finite rank by a group $P$ (called $holonomy$ $group$). Following the notation of Lemma
\ref{l:ramras}, we are fixing $A=L$, $B=\Gamma$ and $Q=P$, imposing a precise choice for these groups. The
$dimension$ of $B$ is the rank of $L$. It is easy to see that $G_p$, studied in the previous section, is of this
form, once $L=\mathbb{Z}^{n-1}$ and $P=C_n$. It is known that
\begin{equation}\label{e:bieberbach2} B_1(2)=\langle a, x, y \ | \ a^2=y, axa^{-1}=x^{-1}, [a,y]=[x,y]=1\rangle \end{equation}
is a $Bieberbach$ $group$ $of$ $dimension$ $2$ $with$ $point$
$group$ $C_2$ and that the groups
\begin{equation}\label{e:bieberbach3} B_1(n) = B_1(2) \times \mathbb{Z}^{n-2} \,\,\,\,
\textrm{for}\,\,\, n\geq 2\end{equation} are $Bieberbach$ $groups$
$of$ $dimension$ $n$ $with$ $point$ $group$ $C_2$. More details can
be found in \cite{rohaidah}. The next two results check (\ref{e:2})
on $B_1(2)$ and $B_1(n)$.

\begin{cor}
 In $B_1(2)$ we have that $f$ is constant to 0.
\end{cor}

\begin{proof}From  \cite[Theorem 4.1]{rohaidah} we have
\begin{equation}B_1(2)\otimes
 B_1(2) = C_2 \times C_4 \times
 \mathbb{Z}^2.\end{equation} Still from \cite{rohaidah} we know that $M(B_1(2))$ is trivial.
Now $f(h(B_1(2)))=h(B_1(2)\otimes B_1(2))-h(B_1(2))=2-2=0$ and the
result follows.
\end{proof}

\begin{cor}
 In $B_1(n)$ we have that $f(h(B_1(n)))=n^2-3n+4$ for all $n>2$.
\end{cor}

\begin{proof}From \cite[Corollary 4.1]{rohaidah} we have
\begin{equation}B_1(n)\otimes
 B_1(n) = C^{2n-3}_2 \times C_4 \times \mathbb{Z}^{(n-1)^2+1}.\end{equation}
Still from \cite{rohaidah} we know that $M(B_1(2))=n-2$ and so it is
nontrivial. Now $f(h(B_1(n)))=h(B_1(n)\otimes
B_1(n))-h(B_1(n))=((n-1)^2+1)-(n-2)=n^2-2n+2-n+2=n^2-3n+4$ and the
result follows.
\end{proof}

In a certain sense Theorem \ref{t:1} (b) forces the growth of (\ref{e:2}) to be equal to that of the Schur
multiplier, when the abelianization of the  group is the direct product of finite cyclic groups. Is this
condition really necessary? Unfortunately, the answer is positive and $B_1(n)$ for $n>2$ shows it.

\begin{cor}
For all $n>2$, $f(h(B_1(n)))$ has not a linear growth but
$h(M(B_1(n)))$ has a linear growth.
\end{cor}

\begin{proof}$f(h(B_1(n)))=n^2-3n+4$ and
$h(M(B_1(n)))=n-2$.
\end{proof}

Recent progresses in \cite{bfm,bm} show that the nonabelian tensor
product of Bieberbach groups has a similar structure with respect to
that of the free solvable groups of finite rank and free nilpotent
groups of finite rank. Therefore we have the following results.

\begin{cor}
Let $F$ be the free group of finite rank $r\geq1$ and $G=F/F^{(d)}$ be the free solvable group of derived length
$d\geq1$ and rank $r$. If $F'$ is periodic, then $f(h(G))\leq \frac{1}{2}r(r-1)$. In particular, if $h(G)=r$,
then the equality holds and $f(h(G))=\frac{1}{2}r(r-1)$.
\end{cor}
\begin{proof} We may apply \cite[Corollary 2.4]{bfm} and so $G\otimes G=\mathbb{Z}^{\f{1}{2}r(r+1)}\times
F'/[F,F^{(d)}]$. Lemma \ref{l:1} implies $h(G\otimes
G)=\frac{1}{2}r(r+1)$. Of course $h(G)\leq r$. Then
$f(h(G))\leq\f{1}{2}r(r+1)-r=\frac{1}{2}r(r-1)$, as claimed.
\end{proof}

\begin{cor}
Let $G$ be the free nilpotent group of rank $r\geq1$ and class $c\geq1$. If $G'$ is periodic, then $f(h(G))\leq
\frac{1}{2}r(r-1)$. In particular, if $h(G)=r$, then the equality holds and $f(h(G))=\f{1}{2}r(r+1)$.
\end{cor}
\begin{proof} Note that nilpotent groups are solvable and so it is meaningful to consider $f(h(G))$.
Applying \cite[Corollary 2.3]{bfm}, $G\otimes G=\mathbb{Z}^{\f{1}{2}r(r+1)}\times
G'$ and the remainder is similar to the previous corollary.
\end{proof}

However, Lemma \ref{l:ramras} imposes the following question, which we leave open in its generality.

\begin{oq}
What is the growth of  $h(\Gamma \otimes \Gamma)$ with respect to
$h(\Gamma)$, where $\Gamma$ is an arbitrary extension of two abelian
groups $A$ and $Q$ as in Lemma \ref{l:ramras}?
\end{oq}

% ------------------------------------------------------------------------

\subsection*{Acknowledgment}

The first and the second author are grateful to the Department of Mathematics and the Ibnu Sina Institute of the
Universiti Teknologi Malaysia for the hospitality in the summer of 2009, when the initial part of this
manuscript was written. We also thank Prof. B. Eick, who suggested \cite{en, eick},  Prof. R. F. Morse and Dr.
P. Moravec, who communicated some inaccuracies in the original version of the present paper. Finally, we
appreciated some email contributions of R. Brown, A. Caranti, D. Feirtenschlager, R. Hartung, M. Horn and D.
Ramras in 2010.

% ------------------------------------------------------------------------

\end{document}